\documentclass[a4paper,12pt]{amsart}
\usepackage[utf8]{inputenc}
\usepackage[english]{babel}
\usepackage{amsaddr}
\usepackage{calrsfs}
\usepackage{amssymb}

\usepackage{float}
\usepackage{graphicx,color}

\makeatletter
\@namedef{subjclassname@2020}{%
	\textup{2020} Mathematics Subject Classification}
\makeatother

\theoremstyle{definition}
\newtheorem{dfn}{Definition}
\theoremstyle{plain}
\newtheorem{thm}{Theorem}
\newtheorem{pro}{Proposition}
\newtheorem{lmm}{Lemma}
\theoremstyle{definition}
\newtheorem{rem}{Remark}

\newcommand{\E}{\mathbb{E}}
\renewcommand{\P}{\mathbb{P}}

\newcommand{\F}{\mathcal{F}}

\newcommand{\K}{\mathrm{K}}
\renewcommand{\L}{\mathrm{L}}

\newcommand{\I}{\mathrm{I}}
\newcommand{\G}{\mathrm{G}}

\newcommand{\e}{\mathrm{e}}
\newcommand{\la}{\langle}
\newcommand{\ra}{\rangle}
\renewcommand{\H}{\mathcal{H}}
\newcommand{\1}{\mathbf{1}}
\renewcommand{\d}{\mathrm{d}}
\newcommand{\p}{\partial}

\begin{document}
\title[Mixed fractional Brownian motion]{Long-range dependent completely correlated mixed fractional Brownian motion}

\date{\today}

\author[Dufitinema]{Josephine Dufitinema}
\address{Department of Mathematics and Statistics, University of Vaasa, P.O. Box 700, FIN-65101 Vaasa, FINLAND}
\email{josephine.dufitinema@uva.fi}

\author[Shokrollahi]{Foad Shokrollahi}
\address{Department of Mathematics and Statistics, University of Vaasa, P.O. Box 700, FIN-65101 Vaasa, FINLAND}
\email{foad.shokrollahi@uva.fi}

\author[Sottinen]{Tommi Sottinen}
\address{Department of Mathematics and Statistics, University of Vaasa, P.O. Box 700, FIN-65101 Vaasa, FINLAND}
\email{tommi.sottinen@iki.fi}

\author[Viitasaari]{Lauri Viitasaari}
\address{Department of Mathematics, Uppsala University, Box 480, 751 06 Uppsala, SWEDEN}
\email{lauri.viitasaari@math.uu.se}


\begin{abstract}
In this paper we introduce the long-range dependent completely correlated mixed fractional Brownian motion (ccmfBm). This is a process that is driven by a mixture of Brownian motion (Bm) and a long-range dependent completely correlated fractional Brownian motion (fBm, ccfBm) that is constructed from the Brownian motion via the Molchan--Golosov representation.  Thus, there is a single Bm driving the mixed process.  
In the short time-scales the ccmfBm behaves like the Bm (it has Brownian H\"older index and quadratic variation). However, in the long time-scales it behaves like the fBm (it has long-range dependence governed by the fBm's Hurst index).  
We provide a transfer principle for the ccmfBm and use it to construct the Cameron--Martin--Girsanov--Hitsuda theorem and prediction formulas. Finally, we illustrate the ccmfBm by simulations.
\end{abstract}

\keywords{Cameron--Martin--Girsanov theorem,
fractional Brownian motion,
fractional Gaussian noises,
long-range dependence,
prediction,
transfer principle}

\subjclass[2020]{Primary 60G22; Secondary 60G25, 60G15}

\maketitle


\section{Introduction}\label{sect:introduction}
Long range dependence have numerous applications in various models and have been a topic of active research, see e.g. monographs \cite{Beran-et-al-2013, Samorodnitsky-2016} and references therein.

The fractional Brownian motion (fBm) is maybe the simplest model for long-range dependence. Indeed, fBm is Gaussian, has stationary increments, and continuous sample paths almost surely. For details on the fBm and it's usage in several applications, we refer to \cite{Mishura-2008}. On the other hand, fBm is not a suitable model if one requires long-range dependence in long time scales and, at the same time, somehow different behaviour in the short time scales. For example in finance, empirical studies suggest long-range dependence while the standard Brownian motion (Bm) would be more suitable model in the short time scales. This can be modelled by a so-called mixed fractional Brownian motion (mfBm) that inherits short time behaviour from the Brownian motion and long time behaviour from the fBm. Cheridito \cite{Cheridito-2001} introduced an independent mfBm as a sum of a Brownian motion and an independent fBm that has received attention in the literature ever since. 

One of the wanted features for a Gaussian model is the so-called transfer principle that allows to reconstruct underlying Brownian motion from the observations. Indeed, once the transfer principle is established, several applications including prediction formulas and parameter estimation become rather straightforward. For this reason transfer principles are actively studied in the literature, see e.g. \cite{Cai-Chigansky-Kleptsyna-2016} in the case of independent mfBm and \cite{Sottinen-Viitasaari-2016b} for the general case in the abstract setting.

In this article we introduce the long-range dependent completely correlated mixed fractional Brownian motion (ccmfBm). We study its basic properties and provide a transfer principle for it. In comparison to the independent mfBm, in our case the transfer principle is explicit allowing easily computable formulas and model fitting in different applications. 

The rest of the paper is organized as follows: 
In Section \ref{sect:defs} we recall what is the fractional Brownian motion (fBm) and recall its basic properties, and introduce the long-range dependent completely correlated mixed fractional Brownian motion (ccmfBm) and state its basic properties.
In Section \ref{sect:transfer_principle} we develop the transfer principle.  In Section \ref{sect:cmg} develop the Cameron--Martin--Girsanov theorem for the ccmfBm by using the transfer principle.  As an application we show how to use the Cameron--Martin--Girsanov in estimating a drift parameter in a ccmfBm model. 
In Section \ref{sect:prediction} we state the prediction formula for the ccmfBm that follows directly from the transfer principle. In Section \ref{sect:simulation} we illustrate the ccmfBm by simulations.  Finally, in Section \ref{sect:conclusions} we summarize our findings and discuss the short-range dependent ccmfBm shortly.

\section{Definitions and Basic Properties}\label{sect:defs}

We recall briefly what is the fractional Brownian motion and state its basic properties.  Then we introduce the long-range dependent completely correlated mixed fractional Brownian motion and state its basic properties.

The fractional Brownian motion (fBm) $B^H={(B^H_t)}_{t\ge0}$ with Hurst index $H\in(0,1)$, introduced by Kolmogorov \cite{Kolmogorov-1940} and christened by Mandelbrot and Van Ness \cite{Mandelbrot-Van-Ness-1968}, is the centered Gaussian process having the covariance function
$$
R_H(t,s) = \frac12\left[t^{2 H} + s^{2 H} - |t-s|^{2 H}\right].
$$
The fBm is the (up to a multiplicative constant) unique centered Gaussian process that has stationary increments and is self-similar with the Hurst index $H$. In particular, if $H=\frac12$, then the fBm is the standard Brownian motion (Bm).
For more details on the fBm, we refer to Biagini et al. \cite{Biagini-Hu-Oksendal-Zhang-2008} and Mishura \cite{Mishura-2008}. 
In this article, we consider a generalization of the fBm where the fBm is mixed with a Bm that generates the fBm.  For other generalizations of the fBm, see e.g. \cite{Ayache-Levy-Vehel-1999,Cheridito-2001,Perrin-Harba-Berzin-Joseph-Iribarren-Bonami-2001,Perrin-Harba-Iribarren-Jennane-2005}.

For $H\in(\frac12,1)$ the fractional Brownian motion is long-range dependent in the sense of the following ``non-stationary'' definition:  
A stochastic process $Y$ is long-range dependent if, for all $t_0\ge0$ and lags $\delta>0$, the incremental covariance
$$
\rho_Y(t_0,\delta;t) =
\E\left[\left(Y_{t_0+\delta}-Y_{t_0}\right)\left(Y_{t+\delta}-Y_t\right)\right]
$$
decay at most following a power law:
$$
\rho_Y(t_0,\delta;t) \ge \frac{C_{t_0,\delta}}{t^{\beta}}
$$
for some $\beta>0$.
Indeed, for fBm with $H>1/2$, we have
\begin{eqnarray*}
\rho_{B^H}(t_0,\delta;t)
&=&
\E\left[\left(B^H_{t_0+\delta}-B^H_{t_0}\right)\left(B^H_{t+\delta}-B^H_t\right)\right] \\
&=&
\frac{1}{\delta^{2H}}\E\left[B^H_1\left(B^H_{\frac{t-t_0}{\delta}+1}-B^H_{\frac{t-t_0}{\delta}}\right)\right] \\
&=&
\frac{1}{\delta^{2H}}
\int_0^1\int_{\frac{t-t_0}{\delta}}^{\frac{t-t_0}{\delta}+1}
\frac{\partial^2 R_H}{\partial u \partial v}(u,v)\, \d u \d v \\
&=&
\frac{H(2H-1)}{\delta^{2H}}
\int_0^1\int_{\frac{t-t_0}{\delta}}^{\frac{t-t_0}{\delta}+1}
(u-v)^{2H-2}\, \d u \d v \\
&\sim&
\frac{H(2H-1)\delta^2}{
\left(t-t_0\right)^{2-2H}}.
\end{eqnarray*}
Here and in what follows, we use the notation $f(t)\sim g(t)$ for  asymptotic equivalence meaning that $\lim f(t)/g(t) =1$.
\begin{rem}
We only consider the case $H\in(1/2,1)$, even though some of the results would be true also for the short-range dependent case where $H\in(0,1/2)$. See Section \ref{sect:conclusions} for further discussion of the short-range completely correlated mixed fractional Brownian motions.
\end{rem}

Denote
$$
c(H) = 
\sqrt{\frac{2H\Gamma(\frac32-H)}{\Gamma(H+\frac12)\Gamma(2-2H)}}\left(H-\frac12\right)
$$
where $\Gamma$
is the Gamma function.
Set
\begin{equation}\label{eq:mg-kernel}
K_H(t,s) =
c(H) \frac{1}{s^{H-\frac12}} \int_s^t \frac{u^{H-\frac12}}{(u-s)^{\frac32-H}}\, \d u.
\end{equation}
Denote $t\wedge s = \min(t,s)$. Then
$$
R_H(t,s) = \int_0^{{t\wedge s}} K_H(t,u)K_H(s,u)\, \d u,
$$
and consequently the following Molchan--Golosov \cite{Molchan-Golosov-1969} representation holds:
\begin{equation}\label{eq:mg}
B^H_t = \int_0^t K_H(t,s)\, \d W_s,
\end{equation}
where $W$ is a Bm.  
We note that the representation \eqref{eq:mg} is of Volterra type meaning that $K_H(t,s)=0$ when $s>t$. 

\begin{dfn}\label{dfn:ccmfbm}
The completely correlated mixed fractional Brownian motion (ccmfBm) is 
$$
X = a W + b B^H
$$
where $a,b\in\mathbb{R}$ with $ab\ne0$, $W$ is a Bm and $B^H$ is a fBm constructed from $W$ via \eqref{eq:mg}.
\end{dfn}

Unlike the mixed fBm with independent summands (see e.g. Cheridito \cite{Cheridito-2001}), the ccmfBm does not have stationary increments.  It however shares some path-properties and long-range dependence with the independent summands fBm. 

\begin{pro}\label{pro:basic}
The ccmfBm is a centered Gaussian process with the covariance function
\begin{eqnarray}\label{eq:cov}
\lefteqn{R(t,s)} \\
&=& \nonumber
a^2({t\wedge s}) + ab\int_0^{{t\wedge s}} \left[K_H(t,u) + K_H(s,u)\right]\,\d u
+ b^2 R_H(t,s).
\end{eqnarray}
Moreover, the ccmfBm
\begin{enumerate}
\item is H\"older continuous with index $1/2$,
\item has quadratic variation $t\mapsto a^2t$,
\item is long-range dependent having the same power law decay in its autocovariance as the fBm.
\end{enumerate}
\end{pro}

\begin{proof}
It is clear that $X$ is centered.  The covariance \eqref{eq:cov} follows from the It\^o-isometry
\begin{eqnarray*}
\E\left[W_t B^H_s\right] 
&=&
\E\left[\int_0^T \1_t(u)\, \d W_u\, \int_0^T K_H(s,u)\, \d W_ u\right] \\
&=& \int_0^{{t\wedge s}} K_H(s,u)\, \d u.
\end{eqnarray*}

\noindent(i)\quad The H\"older index follows from Theorem 1 of \cite{Azmoodeh-Sottinen-Viitasaari-Yazigi-2014} and \eqref{eq:cov}. Indeed, let $s<t$. Then
\begin{eqnarray*}
\lefteqn{
\E\left[(X_t-X_s)^2\right]} \\
&=&
a^2|t-s| + ab\E\left[\left(W_t-W_s\right)\left(B_t^H-B_s^H\right)\right] + b^2|t-s|^{2H}
\end{eqnarray*}
and, by It\^o-isometry and Volterra property,
\begin{eqnarray*}
\lefteqn{\E\left[\left(W_t-W_s\right)\left(B_t^H-B_s^H\right)\right]} \\
&=&
\int_0^T \left(\1_t(u)-\1_s(u)\right)\left(K_H(t,u)-K_H(s,u)\right)\, \d u \\
&=&
\int_s^t K_H(t,u)\, \d u
\\
&=&
K_H(t,u^*)|t-s|
\end{eqnarray*}
for some $u^*\in[s,t]$ as $K_H(t,s)$ is continuous in $s$. 
Since $K_H(t,s) \to 0$ as $s\to t$, we see that
$$
\E\left[(X_t-X_s)^2\right] \sim a^2|t-s|,
$$
which shows that the H\"older index is $1/2$.

\noindent(ii)\quad
The quadratic variation comes from the fact that $B^H$ has zero quadratic variation for $H\in(1/2,1)$ applied to Example 1 of \cite{Bender-Sottinen-Valkeila-2008} (see also \cite{Follmer-1981}).  

\noindent(iii)\quad
Finally, let us show the long-range dependence. 
Now,
\begin{eqnarray*}
\rho_X(t_0,\delta;t)
&=&
\E\left[\left(X_{t_0+\delta}-X_{t_0}\right)\left(X_{t+\delta}-X_t\right)\right] \\
&=&
a^2\E\left[\left(W_{t_0+\delta}-W_{t_0}\right)\left(W_{t+\delta}-W_t\right)\right] \\
& & +  
ab \E\left[\left(W_{t_0+\delta}-W_{t_0}\right)\left(B^H_{t+\delta}-B^H_t\right)\right] \\
& & +
ab \E\left[\left(B^H_{t_0+\delta}-B^H_{t_0}\right)\left(W_{t+\delta}-W_t\right)\right] \\
& & +
b^2 \E\left[\left(B^H_{t_0+\delta}-B^H_{t_0}\right)\left(B^H_{t+\delta}-B^H_t\right)\right] \\
&\sim&
ab \E\left[\left(W_{t_0+\delta}-W_{t_0}\right)\left(B^H_{t+\delta}-B^H_t\right)\right] \\
& & + b^2\rho_{B^H}(t_0,\delta;t).
\end{eqnarray*}
The long-range dependence follows by noting that 
\begin{eqnarray*}
\lefteqn{\E\left[\left(W_{t_0+\delta}-W_{t_0}\right)\left(B^H_{t+\delta}-B^H_t\right)\right]}\\
&=&
\int_{t_0}^{t_0+\delta}
\left[K_H(t+\delta,u)-K_H(t,u)\right]\, \d u \\
&=&
\int_{t_0}^{t_0+\delta}
\int_{t}^{t+\delta}\frac{\partial K_H}{\partial v}(v,u)\, \d v\, \d u \\
&\sim&
\delta^2\frac{\partial K_H}{\partial t}(t,t_0) \\
&=& \delta^2 c(H) \left(\frac{t}{t_0}\right)^{H-\frac12}\frac{1}{(t-t_0)^{\frac32-H}}, 
\end{eqnarray*}
which is the same power decay law, namely $t^{2H-2}$, as the fBm part has:  $\rho_{B^H}(t_0,\delta;t)$.
\end{proof}

Unlike the case of independent summands mfBm with $a=b$ and $H\in(3/4,1)$ (see Cheridito \cite{Cheridito-2001}), the ccmfBm is not equivalent to a Bm in any range of $H\in(1/2,1)$:

\begin{pro}
Let $b\ne 0$. Then the law of the ccmfBm on $[0,T]$ is singular to the law of any multiple of Bm on $[0,T]$.
\end{pro}

\begin{proof}
If follows from Hitsuda \cite{Hitsuda-1968} and the It\^o-isometry that the ccmfBm is equivalent to a multiple of Brownian motion if and only if 
$$
\int_0^T\int_0^T \left[\frac{\partial K_H}{\partial t}(t,s)\right]^2\, \d s\,\d t < \infty.
$$
But 
\begin{eqnarray*}
\lefteqn{
\int_0^T\int_0^T \left[\frac{\partial K_H}{\partial t}(t,s)\right]^2\, \d s\,\d t
} \\
&=&
c_1(\alpha)^2\int_0^T\int_s^T
\left(\frac{t}{s}\right)^{2H-1}\frac{\d t}{(t-s)^{3- 2H}}
\, \d s  \\
&\ge&
c(H)^2 \int_0^T \int_s^T \frac{\d t}{(t-s)^{3-2H}}\, \d s \\
&=&
c(H)^2\int_0^T \int_0^{T-s} \frac{\d u}{u^{3-2H}} \, \d s \\
&=&
\infty
\end{eqnarray*}
due to the non-integrable singularity of $1/u^{3-2H}$ at $u=0$.

Since Gaussian laws are always either equivalent or singular, the claim follows.
\end{proof}

\section{Transfer principle}\label{sect:transfer_principle}

The transfer principle states that from the ccmfBm we can construct a Bm in a non-anticipative way (the inverse transfer principle) and then represent the ccmfBm in a non-anticipative way by using the constructed Bm (the direct transfer principle). We consider processes on a compact time interval $[0,T]$ in this section.  

Let $L^2=L^2([0,T])$ and $\|\cdot\|_2 = \|\cdot\|_{L^2}$.

For a kernel $K\colon[0,T]^2\to\mathbb{R}$ its associated operator is
$$
\K f(t) = \int_0^T f(s)K(t,s)\, \d u.
$$
The adjoint associated operator $\K^*$ of a kernel $K$ is defined by linearly extending the relation
\begin{equation}\label{eq:gen_Kstar}
\K^* \1_t(s) = K(t,s),
\end{equation}
where $\1_t=\1_{[0,t)}$ is the indicator function.

Since the kernel $K_H(t,s)$ is differentiable in $t$ and $K_H(t,t-)=0$, its adjoint associated operator can be written as
\begin{equation}\label{eq:Kstar}
\K_H^* f(t) = \int_t^T f(u)\frac{\p K_H}{\p u}(u,t)\, \d u.
\end{equation}
Indeed, to verify formula \eqref{eq:Kstar} it is enough to check that it satisfies the relation \eqref{eq:gen_Kstar}.

Let $\Lambda$ be the closure of the indicator functions $\1_t$, $t\in [0,T]$, under the inner product generated by the relation
$$
{\la \1_t,\1_s\ra}_\Lambda = R(t,s).
$$
 
Let $\H_1$ be the linear space, or first chaos, of $X$, i.e., the closure of the random variables $X_t$, $t\in[0,T]$, in $L^2(\Omega)$.

For $f\in\Lambda$ the abstract Wiener integral 
$$
\int_0^T f(t)\, \d X_t
$$
is the image of the isometry $\1_t\mapsto X_t$ from $\Lambda$ to $\H_1$.

Denote $L(t,s) = a\1_t(s) + b K_H(t,s)$ and let $\L$ and $\L^*$ be the associated and adjoint associated operators of $L$.  

\begin{lmm}\label{lmm:bounded}
$\L^*$ is a bounded operator on $L^2$ and it can be represented as
\begin{eqnarray}\label{eq:mg-operator}
\L^*f(t) &=& af(t) + b\int_t^T f(u)\frac{\partial K_H}{\partial u}(u,t)\, \d u \\ \nonumber
&=& af(t) + \frac{bc(H)}{t^{H-\frac12}}\int_t^T f(u)\frac{u^{H-\frac12}}{(u-t)^{\frac32-H}}\, \d u.
\end{eqnarray}
\end{lmm}

\begin{proof}
Let us first note that $\L^* = a\I^* + b\K^*_H$, where $\I^*$ is the identity operator and $\K_H^*$ is the adjoint associated operator defined in \eqref{eq:Kstar}, where $K_H(t,s)$ is the Molchan--Golosov kernel $K_H$ defined in \eqref{eq:mg-kernel}.  Since
$$
{\| \L^*f\|}_2 \le |a|{\|f\|}_2 + |b|{\| \K^*_H f\|}_2,
$$
$\L^*$ is bounded in $L^2$, if  $\K^*_H$ is bounded in $L^2$. 
Also, we note that \eqref{eq:mg-operator} is true for step functions.  So, if $\L^*$ is bounded in $L^2$, the formula \eqref{eq:mg-operator} extends to all functions in $L^2$.
Finally we note that $\K^*_H$ is bounded on $L^2$, because (for step functions $f$)
\begin{eqnarray*}
{\|\K^*_H f\|}_2^2
&=&
\int_0^T \left[\K_H^* f(t)\right]^2\, \d t \\
&=&
\int_0^T\int_0^T f(t)f(s)\frac{\partial^2 R_H}{\partial s \partial t}(t,s)\, \d s \d t \\
&=&
H(2H-1)
\int_0^T\int_0^T \frac{f(t)f(s)}{|t-s|^{2-2H}}\, \d s \d t\\
&\le&
H(2H-1)\int_0^T\int_0^T \frac{f(t)^2}{|t-s|^{2-2H}}\, \d s \d t \\
&\le&
H(2H-1)\frac{T^{2H-1}}{H-\frac12}{\|f\|}^2_2, 	
\end{eqnarray*}		
where we have used the elementary estimate 
$$
2|f(t)f(s)| \le f(t)^2 + f(s)^2
$$
and symmetry.
\end{proof}

\begin{lmm}\label{lmm:inversion}
For each $t\in[0,T]$, the integral equation 
\begin{equation}\label{eq:l-inv}
\1_t(s) = a L^{-1}(t,s) + b\int_s^T L^{-1}(t,u)\frac{\partial K_H}{\partial u}(u,s)\, \d u
\end{equation}
admits the unique $L^2$-solution given by
\begin{equation}\label{eq:l-inv-kernel}
L^{-1}(t,s) = \frac{1}{a}\1_t(s)
+ \frac{1}{a}\sum_{k=1}^\infty (-1)^k \left(\frac{b}{a}\right)^k
\gamma_k(t,s)
\end{equation}
where 
$$
\gamma_k(t,s) =
\frac{c(H)^k\Gamma(H-\frac12)^k}{\Gamma\left(k\left(H-\frac12\right)\right)}
\frac{1}{s^{H-\frac12}}\int_s^t u^{H-\frac12}(u-s)^{k(H-\frac12)-1}\, \d u.
$$
\end{lmm}

\begin{proof}
Denote 
$$
G(s,u) = -\frac{b c(H)}{a}
\frac{u^{H-\frac12}}{s^{H-\frac12}(u-s)^{\frac32-H}}.
$$
Then \eqref{eq:l-inv} is the anti-Volterra equation of the second kind
$$
\frac{1}{a}\1_t(s) = L^{-1}(t,s) - \int_s^t L^{-1}(t,u)G(s,u)\, \d u.
$$
Since $\L^* =  a\I^* - a\G$,  Lemma \ref{lmm:bounded} implies that the solution of the equation \eqref{eq:l-inv} is given by the $L^2$-convergent Liouville--Neumann series
\begin{equation}\label{eq:l-inv-kernel-aux}
L^{-1}(t,s) = \sum_{k=1}^\infty \G^k\left[\frac{1}{a}\1_t\right](s),
\end{equation}
where $\G^0$ is the identity operator and $\G^{k+1} =\G\G^{k}$.  
Formula \eqref{eq:l-inv-kernel} follows from formula \eqref{eq:l-inv-kernel-aux} by induction by using the formula
$$
\int_s^u (v-s)^{k\alpha-1}(u-v)^{\alpha-1}\, \d v
=
\frac{\Gamma(k\alpha)\Gamma(\alpha)}{\Gamma((k+1)\alpha)}(u-s)^{(k+1)\alpha-1},
$$ 
where $\alpha=H-\frac12$.
\end{proof}

\begin{rem}
The series \eqref{eq:l-inv-kernel} converges fast.  Indeed, by using the Stirling's approximation
$$
\Gamma\left(k\left(H-\frac12\right)\right)
\sim
\sqrt{2\pi} \left(k\left(H-\frac12\right)\right)^{k\left(H-\frac12\right)-\frac12}
\e^{-k\left(H-\frac12\right)}
$$
and the estimate
\begin{eqnarray*}
\lefteqn{\frac{1}{s^{H-\frac12}}\int_s^t u^{H-\frac12}(u-s)^{k\left(H-\frac12\right)-1}\, \d u} \\
&\le&
\left(\frac{t}{s}\right)^{H-\frac12}\frac{1}{k\left(H-\frac12\right)}(t-s)^{k\left(H-\frac12\right)}
\end{eqnarray*}
we obtain
\begin{eqnarray*}
\gamma_k(t,s) \le
\frac{C^k}
{\left(k\left(H-\frac12\right)\right)^{k\left(H-\frac12\right)+\frac12}}
\left(\frac{t}{s}\right)^{H-\frac12}(t-s)^{k\left(H-\frac12\right)},
\end{eqnarray*}
which also shows that \eqref{eq:l-inv-kernel} converges uniformly for all $s\in[\varepsilon,T]$.
\end{rem}

\begin{figure}[H]\label{fig:picture_summands}
\includegraphics[width=\textwidth]{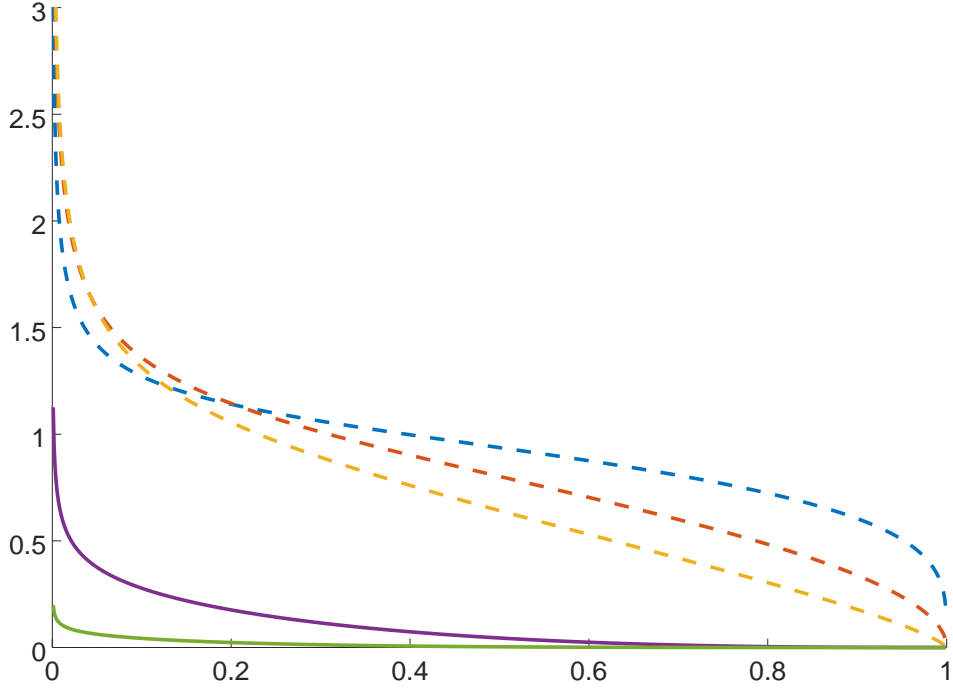}
\caption{The summands $\gamma_k(t,s)$ for $k=1,2,3$ (dashed lines) and $k=10,15$ (solid lines) with $t=1$, $a=b=1$ and $H=0.75$.}
\end{figure}

From Lemma \ref{lmm:inversion} we obtain directly the following basic form of the transfer principle that is the main result of this paper.

\begin{thm}\label{thm:igvp}
Let $L^{-1}(t,s)$ be given by  \eqref{eq:l-inv-kernel}.
The ccmfBm $X$ is an invertible Gaussian Volterra process in the sense that the process $W$ defined as the abstract Wiener integral
$$
W_t = \int_0^t L^{-1}(t,s)\, \d X_s
$$
is a Bm, and the ccmfBm can be reconstructed from it by the Wiener integral
$$
X_t = \int_0^t L(t,s)\, \d W_s,
$$ 
where 
$$
L(t,s) = a\1_t(s) + bK_H(t,s)
$$
and $K_H(t,s)$ is the Molchan--Golosov kernel given by \eqref{eq:mg-kernel}.
\end{thm}

\begin{figure}[H]\label{fig:picture_kernel_L}
\includegraphics[width=\textwidth]{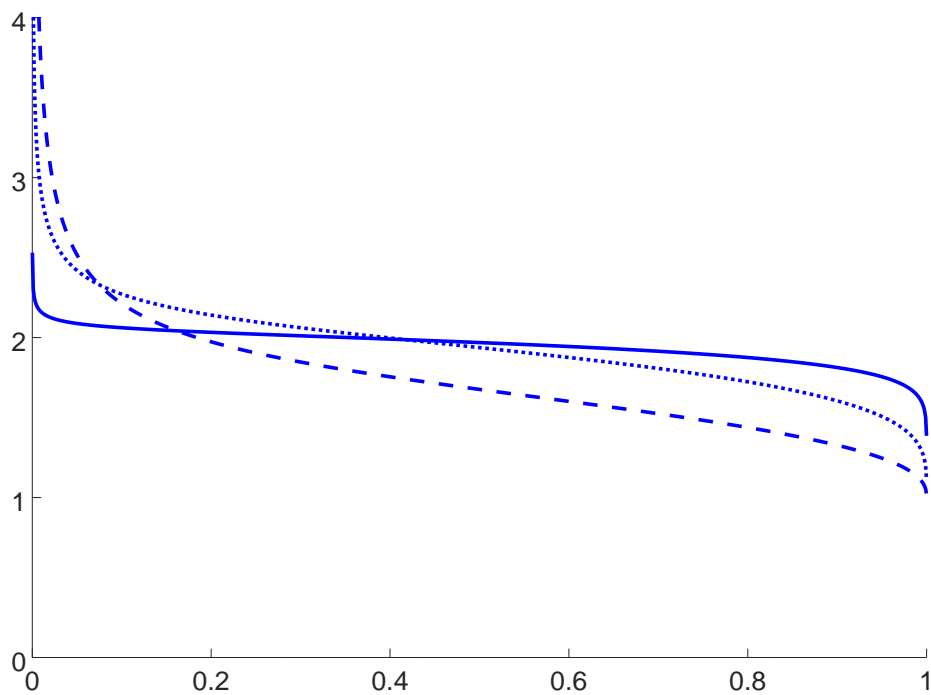}
\caption{The kernels $L(t,s)$ for $t=1$, $a=b=1$ and $H=0.6$ (solid line), $H=0.75$ (dotted line) and $H=0.9$ (dashed line).}
\end{figure}

\begin{figure}[H]
\includegraphics[width=\textwidth]{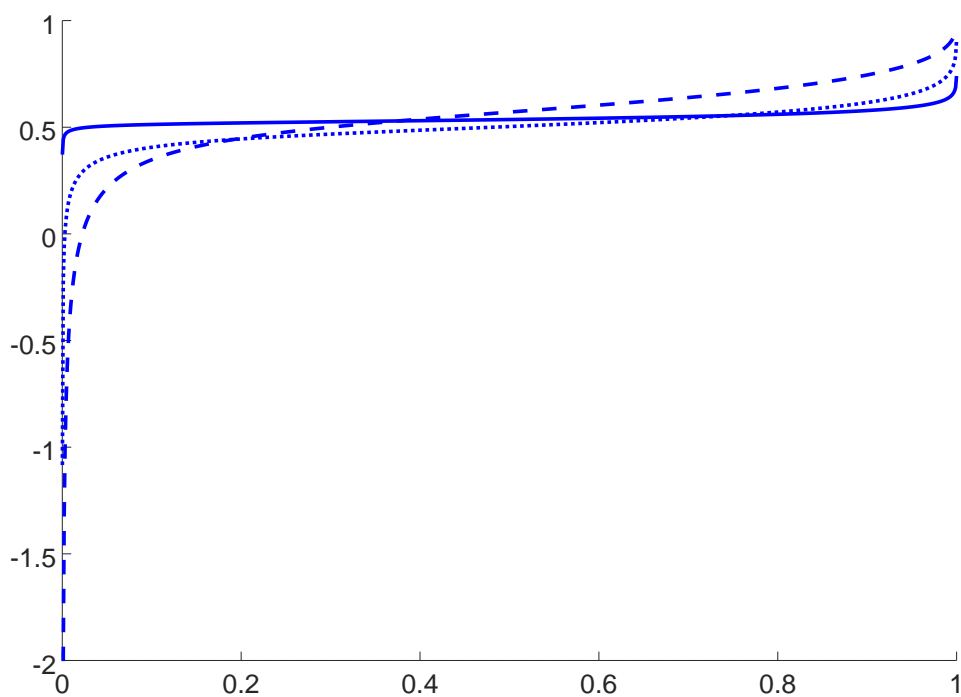}
\caption{The kernels $L^{-1}(t,s)$ for $t=1$, $a=b=1$ and $H=0.6$ (solid line), $H=0.75$ (dotted line) and $H=0.9$ (dashed line).}
\end{figure}

By Lemma \ref{lmm:bounded} and Lemma \ref{lmm:inversion} we have $\Lambda=L^2$. Therefore,
Theorem \ref{thm:igvp} extends to a transfer principle for deterministic integrands on $L^2$ with respect to a Bm and ccmfBm. Actually, we could extend the transfer principle to stochastic integrands by using Malliavin calculus and Skorokhod integration as explained in \cite{Sottinen-Viitasaari-2015,Sottinen-Viitasaari-2016b}.  However, we omit that extension here.

\begin{thm}\label{thm:transfer-principle}
Let $f\in L^2$. Let $X$ and $W$ be connected by Theorem \ref{thm:igvp}.  Then
\begin{eqnarray*}
\int_0^T f(t)\, \d X_t &=& \int_0^T \L^* f(t)\, \d W_t, \\
\int_0^T f(t)\, \d W_t &=& \int_0^T (\L^*)^{-1} f(t)\, \d X_t,
\end{eqnarray*}
where
\begin{eqnarray*}
\L^*f(t) &=&
af(t) + b\int_t^T f(s)\frac{\partial K_H}{\partial s}(s,t)\, \d s, \\
(\L^*)^{-1} f(t) &=& f(t)L^{-1}(T,t)
+ \int_t^T \left[f(s)-f(t)\right]L^{-1}(\d s,t).
\end{eqnarray*}
\end{thm}

\begin{proof}
Using \eqref{eq:l-inv-kernel} we can verify that the kernel $L^{-1}(t,s)$ is of bounded variation in $t$. Moreover, if $f$ is a step function, Theorem \ref{thm:transfer-principle} follows from Theorem \ref{thm:igvp} by linearity and straightforward computations. The claim now follows since $\Lambda=L^2$ and the operator $\L^*$ is an isometry on $L^2$ (see also \cite{Sottinen-Viitasaari-2016b} for details).
\end{proof}	

\begin{rem}
Unlike the independent mfBm introduced by Cheridito \cite{Cheridito-2001}, the ccmfBm does not have stationary increments.  Hence, it could be argued that the independent mfBm is more natural.  However, the inverse transfer principle for the ccmfBm has a rather explicit and convenient form.  For the independent mfBm the inverse transfer principle is implicit and involves solving integral equations numerically, see Cai et al. \cite{Cai-Chigansky-Kleptsyna-2016}.  Therefore, the ccmfBm is more convenient for applications.  We also note that the H\"older continuity properties, quadratic variation, and the long-range dependence of the independent mfBm and ccmfBm are the same.
\end{rem}

\section{Cameron--Martin--Girsanov--Hitsuda theorem}\label{sect:cmg}

In this section we show how the transfer principle of Theorem \ref{thm:transfer-principle} can be used to characterize Gaussian processes that are equivalent in law to the ccmfBm and to provide the corresponding Cameron--Martin--Girsanov theorem in the same way as in \cite{Sottinen-2004,Sottinen-Tudor-2006}.
In this section we consider process on a compact time interval $[0,T]$.

By the Hitsuda representation theorem \cite{Hitsuda-1968} a Gaussian process $\tilde W$ is equivalent to a Bm $W$ if and only if it can be represented
\begin{equation}\label{eq:tildeW}
\tilde W_t = W_t - \int_0^t \int_0^s \ell(s,u)\, \d W_u\, \d s - \int_0^t g(s)\, \d s
\end{equation}
for some $\ell\in L^2([0,T]^2)$ and $g\in L^2([0,T])$.  Here $W$ is a Bm that is constructed from $\tilde W$ by
\begin{equation}\label{eq:W}
W_t = \tilde W_t - \int_0^t\int_0^s \tilde\ell(s,u)\, \d\tilde W_u\, \d s
- \int_0^t\left[g(s)-\int_0^u \tilde\ell(s,u)\, \d u\right]\, \d s,
\end{equation}
where $\tilde\ell$, the resolvent of $\ell$, is given by
\begin{equation}\label{eq:tildel}
\tilde\ell(t,s) = \sum_{k=1}^\infty \ell^{k}(t,s), 
\end{equation}
where
\begin{eqnarray*}
\ell^{1}(t,s) &=& \ell(t,s), \\
\ell^{k+1}(t,s) &=& \int_s^t \ell(t,u)\ell^{k}(u,s)\, \d u.
\end{eqnarray*}
The likelihood ratio of $\tilde W$ over $W$ given observation $\F_t$ on the interval $[0,t]$ is
\begin{eqnarray}
\varphi_t &=& \nonumber
\frac{\d\tilde\P}{\d\P}\Big|\F_t \\
&=& \label{eq:lr1}
\exp\bigg\{\int_0^t\left[\int_0^s \ell(s,u)\d W_u + g(s)\right]\d W_s\\ \nonumber
& &-\frac12\int_0^t\left[\int_0^s\ell(s,u)\d W_u - g(s)\right]^2\d s\bigg\} \\
&=& \label{eq:lr2}
\exp\bigg\{\int_0^t\left[\int_0^s \tilde\ell(s,u)\d\tilde W_u + g(s) -\int_0^s \tilde\ell(s,u)g(u)\, \d u\right]\d \tilde W_s\\ \nonumber
& &-\frac12\int_0^t\left[\int_0^s\tilde\ell(s,u)\d\tilde W_u + g(s) -\int_0^s \tilde\ell(s,u)g(u)\, \d u\right]^2\d s\bigg\}
\end{eqnarray}

Let then $\tilde X$ be a Gaussian process.  Set
\begin{equation}\label{eq:tilde-tp1}
\tilde W_t = \int_0^t L^{-1}(t,s)\, \d\tilde X_s.
\end{equation}
By the transfer principle Theorem \ref{thm:transfer-principle} and Hitsuda representation theorem we then have the following Cameron--Martin--Girsanov--Hitsuda theorem.

\begin{pro}\label{pro:cmgh}
A Gaussian process $\tilde X$ is equivalent in law to a ccmfBm if and only if the process $\tilde W$ given by \eqref{eq:tilde-tp1} is of from \eqref{eq:tildeW} and in this case the process
\begin{equation}\label{eq:tilde-tp2}
X_t = \int_0^t L(t,s)\, \d W_s
\end{equation}
is a ccmfBm, where $W$ is a Bm constructed from $\tilde X$ via \eqref{eq:tilde-tp1} and \eqref{eq:W}.  
The likelihood ratio $\varphi_t$ between $\tilde X$ and $X$ given the observations on $[0,t]$ is given by \eqref{eq:lr1} and \eqref{eq:lr2}.
\end{pro}

The Hitsuda representation can also be given without constructing the intermediate Bm's as follows:

\begin{pro}\label{pro:equivalence}
A Gaussian process $\tilde X$ is equivalent in law to the ccmfBm $X$ if and only if it admits the representation
$$
\tilde X_t = X_t - \int_0^t f(t,s) \, \d X_s
- \L g(t),
$$
where $g\in L^2([0,T])$ and the Voterra kernel $f$ is defined by the equation
\begin{equation}\label{eq:kumma}
\L^* f(t,\cdot)(s) = \L\ell(\cdot,s)(t)
\end{equation}
for some Volterra kernel $\ell\in L^2([0,T]^2)$.
\end{pro}

\begin{rem}
If $X$ is the Bm, then $\L=\I$ is simply the integral operator $\I f(t) = \int_0^t f(s)\d s$ and its adjoint $\L^*=\I^*$ is simply the identity operator.  Thus for the Bm the equation  \eqref{eq:kumma} takes the familiar form of the Hitsuda's representation theorem
$$
f(t,s) = \int_s^t \ell(u,s)\, \d u.
$$
for some $\ell\in L^2([0,T]^2)$.
\end{rem}

The Cameron--Martin--Girsanov likelihood $\varphi_t$ can also be given without constructing the intermediate Bm's by using Skorokhod integrals.  We omit this and instead consider a simple drift estimation case as an application of the Cameron--Martin--Girsanov theorem.

Suppose we want to test the following hypotheses given observations on the time interval $[0,t]$:
\begin{itemize}
\item[$H_0$] The observation comes from a centered ccmfBm $X_s$, $s\in[0,t]$, with known parameters $a,b,H$.
\item[$H_\theta$] The observation comes from a drifting ccmfBm $X_s+\theta s$, $s\in[0,t]$, with known parameters $a,b,H$ but unknown drift $\theta$.
\end{itemize}
Then the likelihood ratio between these hypothesis is
\begin{eqnarray*}
\varphi_t(\theta) &=& \frac{\d\P_\theta}{\d\P_0}\Big|\F_t \\
&=& 
\exp\left\{\theta w_t-\frac{\theta^2}{2} t \right\},
\end{eqnarray*}
where $w$ is constructed from the observations $x$ as 
$$
w_t = \int_0^t L^{-1}(t,s)\, \d x_s
$$
This leads to the maximum likelihood estimator of the drift $\theta$ given observations $x_s$, $s\in[0,t]$:
$$
\hat\theta_t
= \frac{1}{t}\int_0^{t} L^{-1}(t,s)\, \d x_s.
$$

\section{Prediction}\label{sect:prediction}

Let $\F_u^X = \sigma\{X_u; u\le t\}$ denote the $\sigma$-algebra of observing the ccmfBm over the interval $[0,t]$.  Naturally, we are interested in predicting the future, i.e., we are interested in the conditional future probability law of the process $X$ given the information $\F_u^X$.
The transfer principle of Theorem \ref{thm:transfer-principle} provides us these prediction formulas for the ccmfBm in the same way as in \cite{Sottinen-Viitasaari-2017a,Sottinen-Viitasaari-2020}:

\begin{thm}
The conditional process $t\mapsto X_t(u) = X_t|\F_u^X$, $t\ge u$,	is Gaussian with stochastic mean
\begin{eqnarray*}
\hat m_t(u) &=& \E\left[X_t\big| \F_u^X\right] \\
&=& X_u + \int_0^u \Psi(t,s|u)\, \d X_s,
\end{eqnarray*}
where
$$
\Psi(t,s|u) 
=
(\L^*)^{-1}\left[L(t,\cdot)-L(u,\cdot)\right](s);
$$
and with deterministic covariance
\begin{eqnarray*}
\hat R(t,s|u) &=& \mathbb{C}\mathrm{ov}\left[X_t,X_s\big| \F_u^X\right] \\
&=&
R(s,t) - \int_0^u L(t,v)L(s,v)\, \d v.
\end{eqnarray*}	
\end{thm}

\section{Simulation}\label{sect:simulation}

We illustrate the long-range dependent ccmfBm by simulating its paths.  The paths are simulated on $N=500$ equidistant time points $t_k=k/N$ on the interval $[0,1]$ by using the Cholesky decomposition $\mathbf{R}_H = \mathbf{L}_H\mathbf{L}_H^\top$ of the covariance matrix $\mathbf{R}_H(k,j) = R_H(t_k,t_j)$ of the fBm:
\begin{equation}\label{eq:chol_simu}
X_{t_k} = \frac{a}{\sqrt{N}}\sum_{j=1}^{k} \xi_j + b\sum_{j=1}^k \mathbf{L}_H(k,j)\xi_j,
\end{equation}
where $\xi_j$'s are i.i.d. standard random variables.

\begin{rem}
Equation \eqref{eq:chol_simu} provides exact simulation.  The cost of having exact simulation is of course in calculating the Cholesky decomposition.  If one is happy to make some error in the simulation, one can use the Molchan--Golosov representation of the fBm for simulation. One way of doing this is to use the approximation
\begin{equation}\label{eq:simu_approx}
X_{t} \approx
\sum_{j=1}^{[Nt]} \left[a + bN\int_{\frac{j-1}{N}}^{\frac{j}{N}}K_H\left(\frac{[Nt]}{N}, s\right)\, \d s\, \right]\frac{1}{\sqrt{N}}\xi_j
\end{equation}
and approximate the integral above in some way.  If the integral is approximated in an efficient way, then this approach can be fast as it avoids calculating the Cholesky decomposition and instead uses the Molchan--Golosov kernel $K_H$ as a proxy for the Cholesky square root.
See \cite{Sottinen-2001} for more information on the convergence of the approximation \eqref{eq:simu_approx}.
\end{rem}

\begin{rem}
The integral representation
$$
X_t = \int_0^t L(t,s)\, \d W_s
$$
provides us with a series expansion.  Indeed, let $(\tilde e_k)_{k=1}^\infty$ be your favorite orthonormal basis on $L^2([0,T])$ (note that the basis functions $\tilde e_k$ depend on $T$) and let $(\xi_k)_{k=1}^\infty$ be a sequence of i.i.d. standard normal random variables.  Let
\begin{equation}\label{eq:series-integral}
e_k(t) = \int_0^t L(t,s)\tilde e_k(s)\, \d s.
\end{equation}
Then it follows that
\begin{equation}\label{eq:series}
X_t = \sum_{k=1}^\infty e_k(t)\xi_k,
\end{equation}
where the series \eqref{eq:series-integral} converges both in $L^2$ and pointwise.
See \cite{Gilsing-Sottinen-2003} for details and for an explicit series expansion for the fBm that be extended to the ccmfBm in a straightforward manner.
Cutting the series \eqref{eq:series} and approximating the integral \eqref{eq:series-integral} provides us yet another way to approximately simulate the ccmfBm process.
\end{rem}

In figures 4--6 we have plotted simulated paths of the completely correlated mixed fractional Brownian motion (ccmfBm) together with its components, the Brownian motion (Bm) and the completely correlated fractional Brownian motion (ccfBm) that is constructed from the Bm.  The plots were created by using the exact simulation formula \eqref{eq:chol_simu}.

\begin{figure}[H]\label{fig:simu1}
	\includegraphics[width=\textwidth]{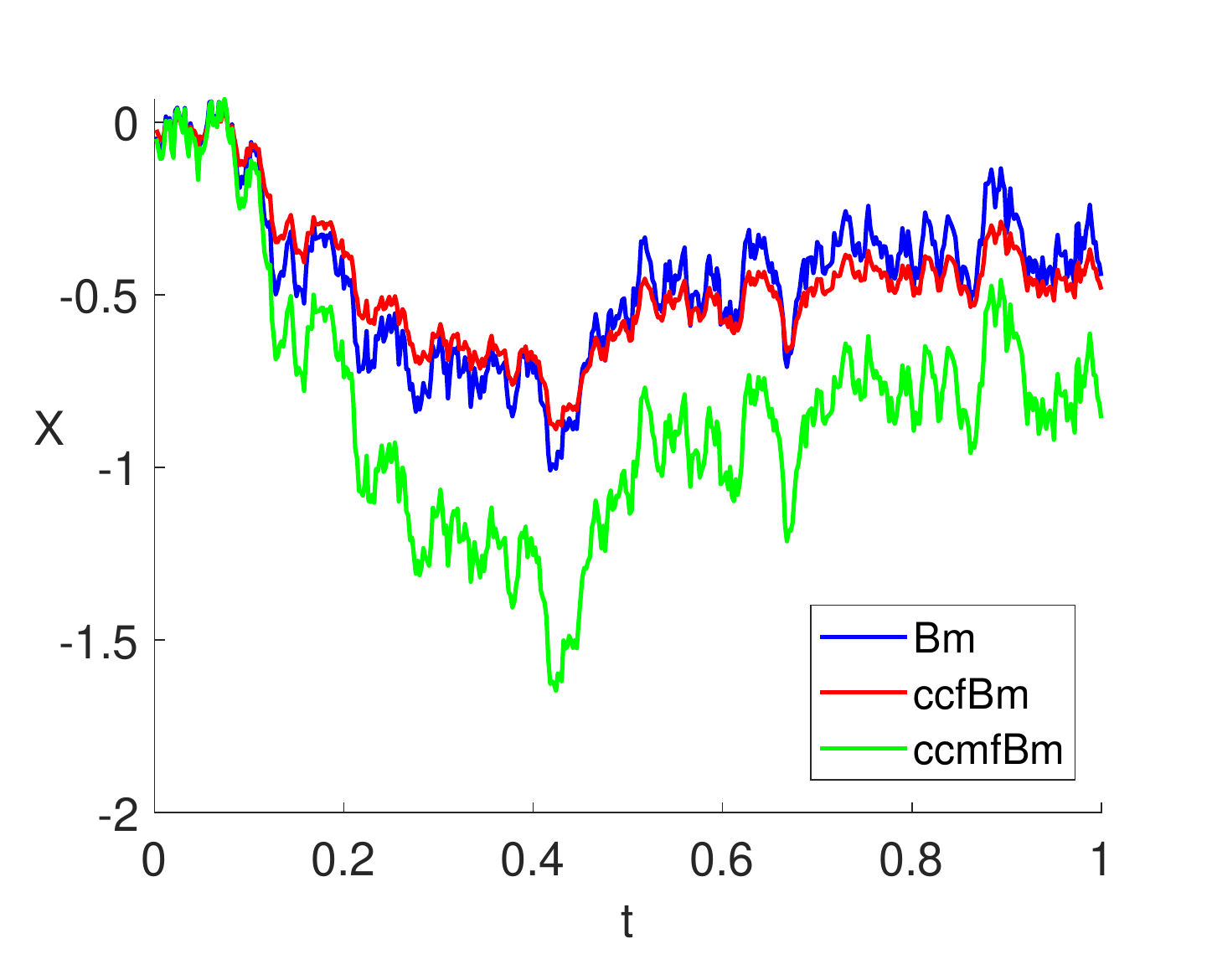}
	\caption{ccmfBm with $a=0.4$, $b=1.4$, $H=0.6$.}
\end{figure}

\begin{figure}[H]\label{fig:simu2}
	\includegraphics[width=\textwidth]{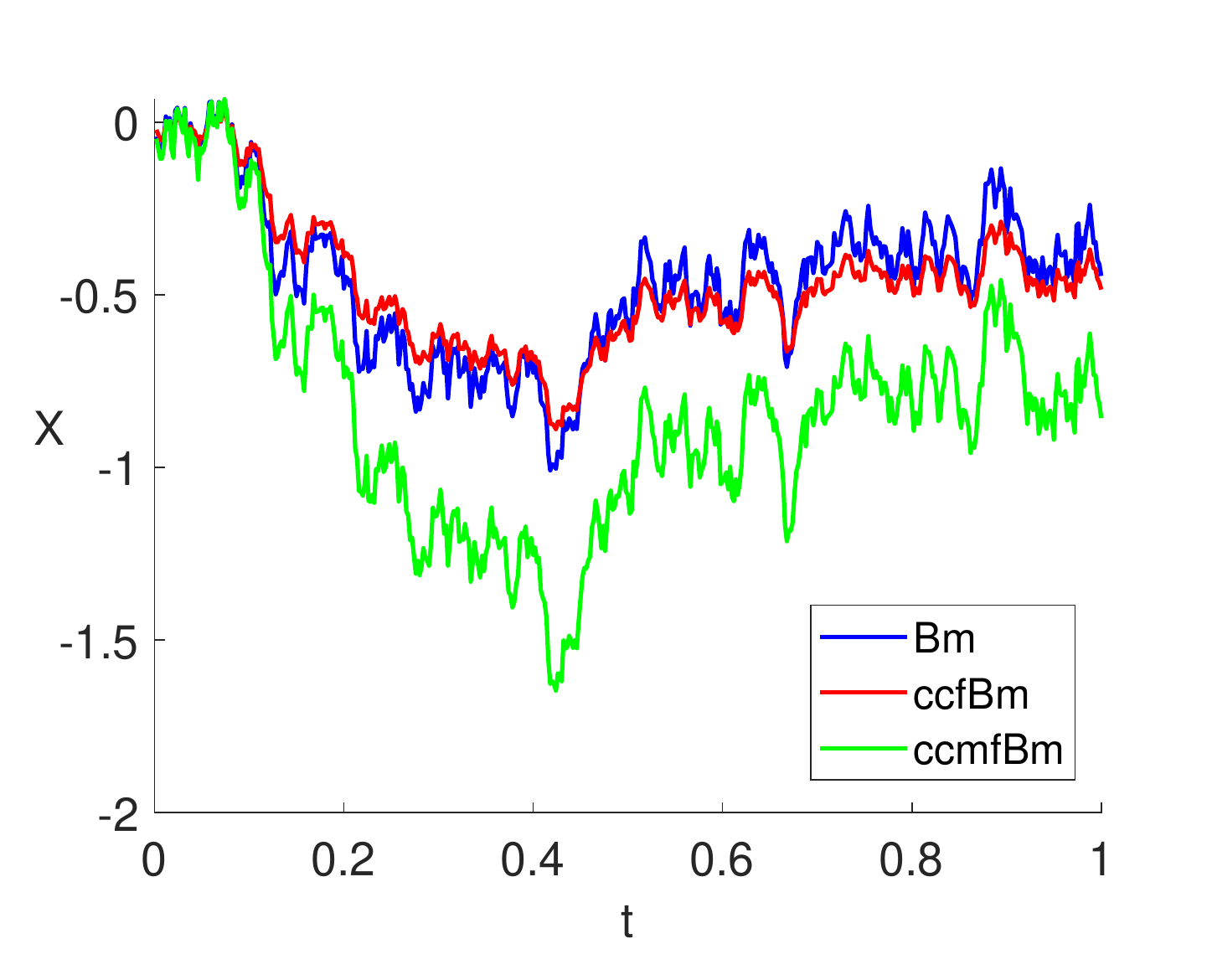}
	\caption{ccmfBm with $a=1$, $b=3$, $H=0.75$.}
\end{figure}

\begin{figure}[H]\label{fig:simu3}
	\includegraphics[width=\textwidth]{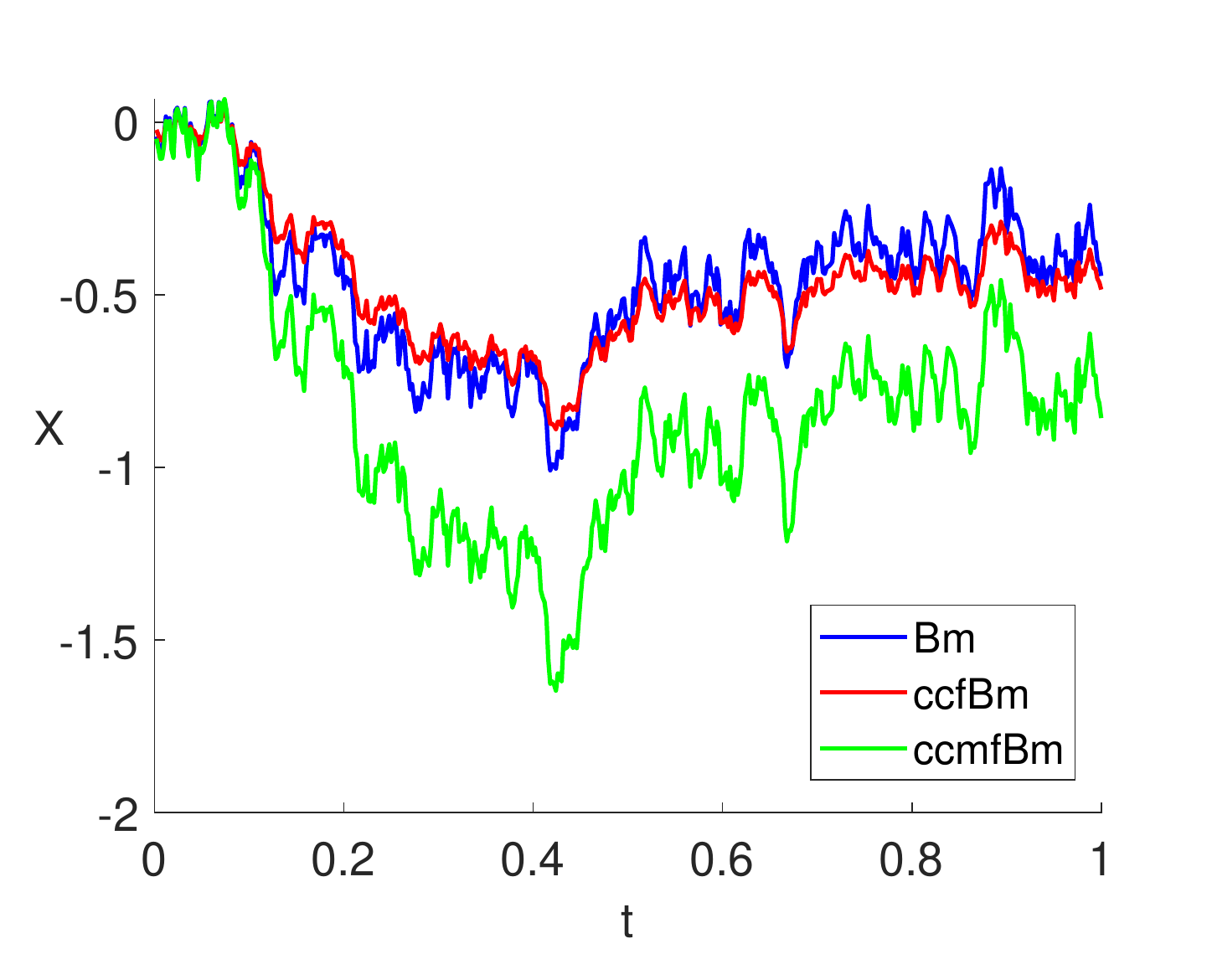}
	\caption{ccmfBm with $a=4$, $b=9$, $H=0.9$.}
\end{figure}

Figures 4--6 illustrate that the H\"older continuity of the ccmfBm is the same as the H\"older continuity of the driving Bm.

\section{Conclusions and Discussion}\label{sect:conclusions}

We have considered the long-range dependent completely correlated mixed fractional Brownian motion (ccmfBm) $X$ that is constructed by using a single Brownian motion (Bm) $W$ and a complete correlated fractional Brownian motion (ccfBm) $B^H$ as
\begin{eqnarray}
X_t &=& a W_t + b B^H_t \nonumber \\
&=&
aW_t + b\int_0^t K_H(t,s)\, \d W_s, \label{eq:tp}
\end{eqnarray}
where $K_H(t,s)$ is the so-called Molchan--Golosov kernel.

We have shown that the short-time path behavior (H\"older continuity and quadratic variation) of the ccmfBm are the same as those of the Bm, but unlike the Bm, the ccmfBm has long-range dependence that is characterized by the fBm part.  We have constructed explicitly the transfer principle for the ccmfBm, i.e., we have constructed explicitly a kernel $L^{-1}$ such that given the ccmfBm $X$ the driving Bm $W$ in \eqref{eq:tp} can be recovered from the equation
$$
W_t = \int_0^t L^{-1}(t,s)\, \d X_s.
$$
We also noted that the transfer principle for ccmfBm has a more convenient form than the corresponding principle of independent mixture mfBm. Indeed, our inverse kernel $L^{-1}$ have a series expansion that converges fast, while for the independent mfBm it is only known that such an inverse kernel exists and it is a solution of a certain integral equation, see \cite{Cai-Chigansky-Kleptsyna-2016}. 
We have considered some applications of the transfer principle like parameter estimation and prediction.  We have illustrated the ccmfBm by simulations.

Finally, let us discuss shortly the short-range dependent ccmfBm.  In the short-range dependent case when $H\in(0,1/2)$ the Molchan--Golosov kernel $K_H$ takes a different form and the adjoint associated operator $\L^*$ of the kernel $L=a\1+bK_H$ is no longer a bounded operator on $L^2$.  Consequently, Lemma \ref{lmm:bounded} is no longer true.  In particular, the integrand space $\Lambda$ will be a strict subset of $L^2$ and the transfer principle of Theorem \ref{thm:transfer-principle} is no longer true without radical modifications. Also, the short time-scale behavior of the short-range dependent ccmfBm will be governed by the fBm part, not the Brownian part:  $H$ will be the H\"older index of the ccmfBm and the quadratic variation will be infinite.

\bibliographystyle{plain}
\bibliography{../../pipliateekki}
\end{document}